\documentclass[12pt]{amsart}
\usepackage{amssymb}
\usepackage{t1enc}
\usepackage[latin2]{inputenc}
\usepackage{verbatim}
\usepackage{amsmath,amsfonts,amssymb,amsthm}
\usepackage[mathcal]{eucal}
\usepackage{enumerate}
\usepackage[centertags]{amsmath}
\usepackage{graphics}

\newtheorem{theorem}{Theorem}
\newtheorem{lemma}{Lemma}

\newtheorem{remark}{Remark}
\newtheorem{example}{Example}
\numberwithin{equation}{subsection}

\begin{document}
\author{L. E. Persson, G. Tephnadze, P. Wall}
\title[Nörlund means]{Some new $\left( H_{p},L_{p}\right) $ type
inequalities of maximal operators of Vilenkin-Nörlund means with
non-decreasing coefficients}
\address{L.-E. Persson, Department of Engineering Sciences and Mathematics,
Lule\aa\ University of Technology, SE-971 87 Lule\aa , Sweden and Narvik
University College, P.O. Box 385, N-8505, Narvik, Norway.}
\email{larserik@ltu.se}
\address{G. Tephnadze, Department of Mathematics, Faculty of Exact and
Natural Sciences, Tbilisi State University, Chavchavadze str. 1, Tbilisi
0128, Georgia}
\email{giorgitephnadze@gmail.com}
\address{P. Wall, Department of Engineering Sciences and Mathematics, Lule%
\aa\ University of Technology, SE-971 87, Lule\aa , Sweden.}
\email{Peter.Wall@ltu.se}
\thanks{The research was supported by a Swedish Institute scholarship,
provided within the framework of the SI Baltic Sea Region Cooperation/Visby
Programme.}
\date{}
\maketitle

\begin{abstract}
In this paper we prove and discuss some new $\left( H_{p},L_{p}\right) $
type inequalities of maximal operators of Vilenkin-Nörlund means with
non-decreasing coefficients. We also apply these inequalities to prove
strong convergence theorems of such Vilenkin-Nörlund means. These
inequalities are the best possible in a special sense. As applications, both
some well-known and new results are pointed out.
\end{abstract}



\section{Introduction}

\bigskip\ The definitions and notations used in this introduction can be
found in our next Section. In the one-dimensional case the weak (1,1)-type
inequality for the maximal operator of Fejér means 
\begin{equation*}
\sigma ^{\ast }f:=\sup_{n\in \mathbb{N}}\left\vert \sigma _{n}f\right\vert
\end{equation*}%
can be found in Schipp \cite{Sc} for Walsh series and in Pál, Simon \cite{PS}
for bounded Vilenkin series. Fujji \cite{Fu} and Simon \cite{Si2} verified
that $\sigma ^{\ast }$ is bounded from $H_{1}$ to $L_{1}$. Weisz \cite{We2}
generalized this result and proved boundedness of $\ \sigma ^{\ast }$ \ from
the martingale space $H_{p}$ to the space $L_{p},$ for $p>1/2$. Simon \cite%
{Si1} gave a counterexample, which shows that boundedness does not hold for $%
0<p<1/2.$ A counterexample for $p=1/2$ was given by Goginava \cite{Go} (see
also \cite{tep1}). Moreover, Weisz \cite{We4} proved that the maximal
operator of the Fejér means $\sigma ^{\ast }$ is bounded from the Hardy
space $H_{1/2}$ to the space $weak-L_{1/2}$. In \cite{tep2} and \cite{tep3}
it was proved that the weighted maximal operator of Fejér means \ 
\begin{equation*}
\widetilde{\sigma }_{p}^{\ast }f:=\sup_{n\in \mathbb{N}_{+}}\frac{\left\vert
\sigma _{n}f\right\vert }{\left( n+1\right) ^{1/p-2}\log ^{2\left[ 1/2+p%
\right] }\left( n+1\right) }
\end{equation*}%
is bounded from the Hardy space $H_{p}$ to the space $L_{p},$ when $0<p\leq
1/2.$ Moreover, the rate of the weights $\left\{ 1/\left( n+1\right)
^{1/p-2}\log ^{2\left[ p+1/2\right] }\left( n+1\right) \right\}
_{n=1}^{\infty }$ in $n$-th Fejér mean was given exactly.

Móricz and Siddiqi \cite{Mor} investigated the approximation properties of
some special Nörlund means of Walsh-Fourier series of $L_{p}$ function in
norm. \ In the two-dimensional case approximation properties of Nörlund
means was considered by Nagy (see \cite{nag}-\cite{nagy}). In \cite{ptw} it
was proved that the maximal operator of Nörlund means 
\begin{equation*}
t^{\ast }f:=\sup_{n\in 
\mathbb{N}
}\left\vert t_{n}f\right\vert
\end{equation*}%
with non-decreasing coefficients is bounded from the Hardy space $H_{1/2}$
to the space $weak-L_{1/2}$. Moreover, there exists a martingale and Nörlund
means, with non-decreasing coefficients, such that it is not bounded from
the Hardy space $H_{p}$ to the space $weak-L_{p},$ when $0<p<1/2.$

It is well-known that Vilenkin systems do not form bases in the space $L_{1}$%
. Moreover, there is a function in the Hardy space $H_{1}$, such that the
partial sums of $f$ are not bounded in $L_{1}$-norm. Simon \cite{Si4} proved
that there exists an absolute constant $c_{p},$ depending only on $p,$ such
that the inequality 
\begin{equation*}
\frac{1}{\log ^{\left[ p\right] }n}\overset{n}{\underset{k=1}{\sum }}\frac{%
\left\Vert S_{k}f\right\Vert _{p}^{p}}{k^{2-p}}\leq c_{p}\left\Vert
f\right\Vert _{H_{p}}^{p}\ \ \ \left( 0<p\leq 1\right)
\end{equation*}%
holds for all $f\in H_{p}$ and $n\in \mathbb{N}_{+},$ where $\left[ p\right] 
$ denotes the integer part of $p.$ For $p=1$ analogous results with respect
to more general systems were proved in \cite{b} and \cite{Ga1} and for $%
0<p<1 $ another proof can be found in \cite{tep8}.

In \cite{bt1} it was proved that there exists an absolute constant $c_{p}$,
depending only on $p$, such that the inequality 
\begin{equation}
\frac{1}{\log ^{\left[ 1/2+p\right] }n}\overset{n}{\underset{k=1}{\sum }}%
\frac{\left\Vert \sigma _{k}f\right\Vert _{p}^{p}}{k^{2-2p}}\leq
c_{p}\left\Vert f\right\Vert _{H_{p}}^{p}\ \ \ \left( 0<p\leq 1/2,\
n=2,3,\dots \right) .  \label{3cc}
\end{equation}%
holds. An analogous result for the Walsh system can be found in \cite{tep9}.

In this paper we derive some new $(H_{p},L_{p})$-type inequalities for
weighted maximal operators of Nörlund means with non-decreasing
coefficients. Moreover, we prove strong convergence theorems of such Nörlund
means.

This paper is organized as follows: In order not to disturb our discussions
later on some definitions and notations are presented in Section 2. The main
results and some of its consequences can be found in Section 3. For the
proofs of the main results we need some auxiliary Lemmas, some of them are
new and of independent interest. These results are presented in Section 4.
The detailed proofs are given in Section 5.

\section{Definitions and Notation}

Denote by $%
\mathbb{N}
_{+}$ the set of the positive integers, $%
\mathbb{N}
:=%
\mathbb{N}
_{+}\cup \{0\}.$ Let $m:=(m_{0,}$ $m_{1},...)$ be a sequence of the positive
integers not less than 2. Denote by 
\begin{equation*}
Z_{m_{k}}:=\{0,1,...,m_{k}-1\}
\end{equation*}
the additive group of integers modulo $m_{k}$.

Define the group $G_{m}$ as the complete direct product of the groups $%
Z_{m_{i}}$ with the product of the discrete topologies of $Z_{m_{j}}`$s.

The direct product $\mu $ of the measures 
\begin{equation*}
\mu _{k}\left( \{j\}\right) :=1/m_{k}\text{ \ \ \ }(j\in Z_{m_{k}})
\end{equation*}%
is the Haar measure on $G_{m_{\text{ }}}$with $\mu \left( G_{m}\right) =1.$

In this paper we discuss bounded Vilenkin groups,\textbf{\ }i.e. the case
when $\sup_{n}m_{n}<\infty .$

The elements of $G_{m}$ are represented by sequences 
\begin{equation*}
x:=\left( x_{0},x_{1},...,x_{j},...\right) ,\ \left( x_{j}\in
Z_{m_{j}}\right) .
\end{equation*}

Set $e_{n}:=\left( 0,...,0,1,0,...\right) \in G,$ the $n-$th coordinate of
which is 1 and the rest are zeros $\left( n\in 
\mathbb{N}
\right) .$

It is easy to give a basis for the neighborhoods of $G_{m}:$ 
\begin{equation*}
I_{0}\left( x\right) :=G_{m},\text{ \ }I_{n}(x):=\{y\in G_{m}\mid
y_{0}=x_{0},...,y_{n-1}=x_{n-1}\},
\end{equation*}%
where $x\in G_{m},$ $n\in 
\mathbb{N}
.$

If we define $I_{n}:=I_{n}\left( 0\right) ,$\ for \ $n\in \mathbb{N}$ and $\ 
\overline{I_{n}}:=G_{m}$ $\backslash $ $I_{n},$ then%
\begin{equation}
\overline{I_{N}}=\left( \overset{N-2}{\underset{k=0}{\bigcup }}\overset{N-1}{%
\underset{l=k+1}{\bigcup }}I_{N}^{k,l}\right) \bigcup \left( \underset{k=1}{%
\bigcup\limits^{N-1}}I_{N}^{k,N}\right) ,  \label{1.1}
\end{equation}%
where 
\begin{equation*}
I_{N}^{k,l}:=\left\{ 
\begin{array}{l}
\text{ }I_{N}(0,...,0,x_{k}\neq 0,0,...,0,x_{l}\neq 0,x_{l+1\text{ }%
},...,x_{N-1\text{ }},...),\text{ \ for }k<l<N, \\ 
\text{ }I_{N}(0,...,0,x_{k}\neq 0,0,...,,x_{N-1\text{ }}=0,\text{ }x_{N\text{
}},...),\text{ \ \ \ \ \ \ \ \ \ \ \ \ \ \ \ for }l=N.%
\end{array}%
\text{ }\right.
\end{equation*}

\bigskip If we define the so-called generalized number system based on $m$
in the following way : 
\begin{equation*}
M_{0}:=1,\ M_{k+1}:=m_{k}M_{k}\,\,\,\ \ (k\in 
\mathbb{N}
),
\end{equation*}%
then every $n\in 
\mathbb{N}
$ can be uniquely expressed as $n=\sum_{j=0}^{\infty }n_{j}M_{j},$ where $%
n_{j}\in Z_{m_{j}}$ $(j\in 
\mathbb{N}
_{+})$ and only a finite number of $n_{j}`$s differ from zero.

We introduce on $G_{m}$ an orthonormal system which is called the Vilenkin
system. At first, we define the complex-valued function $r_{k}\left(
x\right) :G_{m}\rightarrow 
\mathbb{C}
,$ the generalized Rademacher functions, by%
\begin{equation*}
r_{k}\left( x\right) :=\exp \left( 2\pi ix_{k}/m_{k}\right) ,\text{ }\left(
i^{2}=-1,x\in G_{m},\text{ }k\in 
\mathbb{N}
\right) .
\end{equation*}

Next, we define the Vilenkin system$\,\,\,\psi :=(\psi _{n}:n\in 
\mathbb{N}
)$ on $G_{m}$ by: 
\begin{equation*}
\psi _{n}(x):=\prod\limits_{k=0}^{\infty }r_{k}^{n_{k}}\left( x\right)
,\,\,\ \ \,\left( n\in 
\mathbb{N}
\right) .
\end{equation*}

Specifically, we call this system the Walsh-Paley system when $m\equiv 2.$

The norms (or quasi-norms) of the spaces $L_{p}(G_{m})$ and $%
weak-L_{p}\left( G_{m}\right) $ $\left( 0<p<\infty \right) $ are
respectively defined by 
\begin{equation*}
\left\Vert f\right\Vert _{p}^{p}:=\int_{G_{m}}\left\vert f\right\vert
^{p}d\mu ,\text{ }\left\Vert f\right\Vert _{weak-L_{p}}^{p}:=\underset{%
\lambda >0}{\sup }\lambda ^{p}\mu \left( f>\lambda \right) <+\infty .
\end{equation*}%
\qquad

The Vilenkin system is orthonormal and complete in $L_{2}\left( G_{m}\right) 
$ (see \cite{Vi}).

Now, we introduce analogues of the usual definitions in Fourier-analysis. If 
$f\in L_{1}\left( G_{m}\right) $ we can define Fourier coefficients, partial
sums of the Fourier series and Dirichlet kernels with respect to the
Vilenkin system in the usual manner: 
\begin{equation*}
\widehat{f}\left( n\right) :=\int_{G_{m}}f\overline{\psi }_{n}d\mu \,\ \ \ \
\ \ \,\left( n\in 
\mathbb{N}
\right) ,
\end{equation*}%
\begin{equation*}
S_{n}f:=\sum_{k=0}^{n-1}\widehat{f}\left( k\right) \psi _{k},\text{ \ \ }%
D_{n}:=\sum_{k=0}^{n-1}\psi _{k\text{ }},\text{ \ \ }\left( n\in 
\mathbb{N}
_{+}\right) ,
\end{equation*}%
respectively.

The $\sigma $-algebra generated by the intervals $\left\{ I_{n}\left(
x\right) :x\in G_{m}\right\} $ will be denoted by $\digamma _{n}\left( n\in 
\mathbb{N}
\right) .$ Denote by $f=\left( f^{\left( n\right) },n\in 
\mathbb{N}
\right) $ a martingale with respect to $\digamma _{n}\left( n\in 
\mathbb{N}
\right) .$ (for details see e.g. \cite{We1}).

The maximal function of a martingale $f$ \ is defined by 
\begin{equation*}
f^{\ast }=\sup_{n\in 
\mathbb{N}
}\left\vert f^{(n)}\right\vert .
\end{equation*}

For $0<p<\infty $ \ the Hardy martingale spaces $H_{p}$ consist of all
martingales $f$ for which 
\begin{equation*}
\left\Vert f\right\Vert _{H_{p}}:=\left\Vert f^{\ast }\right\Vert
_{p}<\infty .
\end{equation*}

If $f=\left( f^{\left( n\right) },n\in 
\mathbb{N}
\right) $ is a martingale, then the Vilenkin-Fourier coefficients must be
defined in a slightly different manner: 
\begin{equation*}
\widehat{f}\left( i\right) :=\lim_{k\rightarrow \infty
}\int_{G_{m}}f^{\left( k\right) }\overline{\psi }_{i}d\mu .
\end{equation*}

Let $\{q_{k}:k\geq 0\}$ be a sequence of nonnegative numbers. The $n$-th Nö%
rlund mean $t_{n}$ for a Fourier series of $f$ \ is defined by 
\begin{equation}
t_{n}f=\frac{1}{Q_{n}}\overset{n}{\underset{k=1}{\sum }}q_{n-k}S_{k}f,
\label{nor}
\end{equation}%
where $Q_{n}:=\sum_{k=0}^{n-1}q_{k}.$ \ 

We always assume that $q_{0}>0$ and $\ \lim_{n\rightarrow \infty
}Q_{n}=\infty .$\ \ In this case it is well-known that the summability
method generated by $\{q_{k}:k\geq 0\}$ is regular if and only if 
\begin{equation*}
\lim_{n\rightarrow \infty }\frac{q_{n-1}}{Q_{n}}=0.
\end{equation*}

Concerning this fact and related basic results we refer to \cite{moo}. In
this paper we consider regular Nörlund means only.

If $q_{k}\equiv 1,$ we respectively define the Fejér means $\sigma _{n}$ and
Kernels $K_{n}$ as follows: 
\begin{equation*}
\sigma _{n}f:=\frac{1}{n}\sum_{k=1}^{n}S_{k}f\,,\text{ \ \ }K_{n}:=\frac{1}{n%
}\sum_{k=1}^{n}D_{k}.
\end{equation*}

It is well-known that (see \cite{AVD})%
\begin{equation}
n\left\vert K_{n}\right\vert \leq c\sum_{l=0}^{\left\vert n\right\vert
}M_{l}\left\vert K_{M_{l}}\right\vert  \label{fn5}
\end{equation}%
\thinspace and%
\begin{equation}
\left\Vert K_{n}\right\Vert _{1}\leq c<\infty .  \label{fn6}
\end{equation}

Denote%
\begin{equation*}
\log ^{(0)}x=x\text{ \ \ and \ }\log ^{(\beta )}x:=\overset{\beta \text{
times}}{\overbrace{\log ...\log }}x,\text{ for }\beta \in 
\mathbb{N}
_{+}.
\end{equation*}

Let $\alpha \in 
\mathbb{R}
_{+},$ $\beta \in 
\mathbb{N}
_{+}$ and $\left\{ q_{k}=\log ^{\left( \beta \right) }k^{\alpha }:k\geq
0\right\} .$ Then we get the class of Nörlund means, with non-decreasing
coefficients: 
\begin{equation*}
\theta _{n}f:=\frac{1}{Q_{n}}\sum_{k=1}^{n}\log ^{\left( \beta \right)
}\left( n-k\right) ^{\alpha }S_{k}f,
\end{equation*}%
where%
\begin{equation*}
Q_{n}=\sum_{k=1}^{n-1}\log ^{\left( \beta \right) }\left( n-k\right)
^{\alpha }=\sum_{k=1}^{n-1}\log ^{\left( \beta \right) }k^{\alpha }=\log
\prod\limits_{k=1}^{n-1}\log ^{\left( \beta -1\right) }k^{\alpha }
\end{equation*}%
\begin{equation*}
\geq \log \left( \log ^{\left( \beta -1\right) }\left( \frac{n-1}{2}\right)
^{\alpha }\right) ^{\frac{\left( n-1\right) }{2}}\geq \frac{n}{4}\log \log
^{\left( \beta -1\right) }\left( \frac{n-1}{2}\right) ^{\alpha }\sim n\log
^{\left( \beta \right) }n^{\alpha }.
\end{equation*}

It follows that%
\begin{equation*}
\frac{q_{n-1}}{Q_{n}}\leq \frac{c\log ^{\left( \beta \right) }\left(
n-1\right) ^{\alpha }}{n\log ^{\beta }n^{\alpha }}=O\left( \frac{1}{n}%
\right) \rightarrow 0,\text{ as \ }n\rightarrow \infty .
\end{equation*}

Finally, we say that a bounded measurable function $a$ is a p-atom, if there
exists a interval $I$, such that%
\begin{equation*}
\int_{I}ad\mu =0,\text{ \ \ }\left\Vert a\right\Vert _{\infty }\leq \mu
\left( I\right) ^{-1/p},\text{ \ \ supp}\left( a\right) \subset I.
\end{equation*}

\section{The Main Results and Applications}

Our first main result reads:

\begin{theorem}
\label{theorem2fejerstrong}a) Let $0<p<1/2,$ $f\in H_{p}$ and $\{q_{k}:k\geq
0\}$ be a sequence of non-decreasing numbers. Then there exists an absolute
constant $c_{p},$ depending only on $p,$ such that the inequality%
\begin{equation*}
\overset{\infty }{\underset{k=1}{\sum }}\frac{\left\Vert t_{k}f\right\Vert
_{p}^{p}}{k^{2-2p}}\leq c_{p}\left\Vert f\right\Vert _{H_{p}}^{p}
\end{equation*}%
holds.

b)Let $f\in H_{1/2}$ and $\{q_{k}:k\geq 0\}$ be a sequence of non-decreasing
numbers, satisfying the condition 
\begin{equation}
\frac{q_{n-1}}{Q_{n}}=O\left( \frac{1}{n}\right) ,\text{ \ \ as \ \ }\
n\rightarrow \infty .  \label{fn01}
\end{equation}%
Then there exists an absolute constant $c,$ such that the inequality%
\begin{equation}
\frac{1}{\log n}\overset{n}{\underset{k=1}{\sum }}\frac{\left\Vert
t_{k}f\right\Vert _{1/2}^{1/2}}{k}\leq c\left\Vert f\right\Vert
_{H_{1/2}}^{1/2}  \label{7nor}
\end{equation}%
holds.
\end{theorem}

\begin{example}
Let $0<p\leq 1/2,$ $f\in H_{p}$ and $\{q_{k}:k\geq 0\}$ be a sequence of
non-decreasing numbers, such that%
\begin{equation*}
\sup_{n}q_{n}<c<\infty .
\end{equation*}%
Then 
\begin{equation*}
\frac{q_{n-1}}{Q_{n}}\leq \frac{c}{Q_{n}}\leq \frac{c}{q_{0}n}=\frac{c_{1}}{n%
}=O\left( \frac{1}{n}\right) ,\text{ as }n\rightarrow 0,
\end{equation*}%
i.e. condition (\ref{fn01}) is satisfied and for such Nörlund means there
exists an absolute constant $c,$ such that the inequality (\ref{7nor}) holds.
\end{example}

\begin{example}
Let $0<p\leq 1/2$ and $f\in H_{p}.$ Then there exists absolute constant $%
c_{p},$ depending only on $p,$ such that the following inequality holds: 
\begin{equation*}
\frac{1}{\log ^{\left[ 1/2+p\right] }n}\overset{n}{\underset{k=1}{\sum }}%
\frac{\left\Vert \sigma _{k}f\right\Vert _{p}^{p}}{k^{2-2p}}\leq
c_{p}\left\Vert f\right\Vert _{H_{p}}^{p}.
\end{equation*}
\end{example}

\begin{remark}
This result for the Walsh system can be found in \cite{tep9} and for any
bounbed Vilenkin system in \cite{bt1}.
\end{remark}

We have already considered the case when the sequence $\{q_{k}:k\geq 0\}$ is
bounded. Now, we consider some Nörlund means, which are generated by a
unbounded sequence $\{q_{k}:k\geq 0\}.$

\begin{example}
Let $0<p\leq 1/2$ and $f\in H_{p}.$ Then there exists an absolute constant $%
c_{p},$ depending only on $p,$ such that the following inequality holds: 
\begin{equation*}
\frac{1}{\log ^{\left[ 1/2+p\right] }n}\overset{n}{\underset{k=1}{\sum }}%
\frac{\left\Vert \theta _{k}f\right\Vert _{p}^{p}}{k^{2-2p}}\leq
c_{p}\left\Vert f\right\Vert _{H_{p}}^{p}.
\end{equation*}
\end{example}

Up to now we have considered strong convergence theorems in the case $%
0<p\leq 1/2,$\ but in our next main result we consider boundedness of
weighed maximal operators of Nörlund means when $0<p\leq 1/2,$ and now
without any restriction like (\ref{fn01}).

\begin{theorem}
\label{theorem3fejermax2}Let $0<p\leq 1/2,$ $f\in H_{p}$ and $\{q_{k}:k\geq
0\}$ be a sequence of non-decreasing numbers. Then the maximal operator \ 
\begin{equation*}
\widetilde{t}_{p}^{\ast }f:=\sup_{n\in \mathbb{N}_{+}}\frac{\left\vert
t_{n}f\right\vert }{\left( n+1\right) ^{1/p-2}\log ^{2\left[ 1/2+p\right]
}\left( n+1\right) }
\end{equation*}%
is bounded from the Hardy space $H_{p}$ to the space $L_{p}.$
\end{theorem}

\begin{example}
Let $0<p\leq 1/2,$ $f\in H_{p}$ and $\{q_{k}:k\geq 0\}$ be a sequence of
non-decreasing numbers. Then the maximal operator \ 
\begin{equation*}
\widetilde{\sigma }_{p}^{\ast }f:=\sup_{n\in \mathbb{N}_{+}}\frac{\left\vert
\sigma _{n}f\right\vert }{\left( n+1\right) ^{1/p-2}\log ^{2\left[ 1/2+p%
\right] }\left( n+1\right) }
\end{equation*}%
is bounded from the Hardy space $H_{p}$ to the space $L_{p}.$
\end{example}

\begin{remark}
This result for the Walsh system when $p=1/2$ can be found in \cite{gog2}.
Later on, it was generalized for bounded Vilenkin systems in \cite{tep2}.
The case $0<p<1/2$ can be found in \cite{tep3}. Analogous results with
respect to Walsh-Kachmarz systems were considered in \cite{gn} for $p=1/2$
and in \cite{tep4} for $0<p<1/2.$
\end{remark}

\begin{example}
Let $0<p\leq 1/2,$ $f\in H_{p}$ and $\{q_{k}:k\geq 0\}$ be a sequence of
non-decreasing numbers. Then the maximal operator%
\begin{equation*}
\widetilde{\theta }_{p}^{\ast }f:=\sup_{n\in \mathbb{N}_{+}}\frac{\left\vert
\theta _{n}f\right\vert }{\left( n+1\right) ^{1/p-2}\log ^{2\left[ 1/2+p%
\right] }\left( n+1\right) }
\end{equation*}%
is bounded from the Hardy space $H_{p}$ to the space $L_{p}.$
\end{example}

\section{Auxiliary lemmas}

We need the following auxiliary Lemmas:

\begin{lemma}[see e.g. \protect\cite{We3}]
\label{lemma2.1} A martingale $f=\left( f^{\left( n\right) },n\in 
\mathbb{N}
\right) $ is in $H_{p}\left( 0<p\leq 1\right) $ if and only if there exists
a sequence $\left( a_{k},k\in 
\mathbb{N}
\right) $ of p-atoms and a sequence $\left( \mu _{k},k\in 
\mathbb{N}
\right) $ of real numbers such that, for every $n\in 
\mathbb{N}
,$ 
\begin{equation}
\qquad \sum_{k=0}^{\infty }\mu _{k}S_{M_{n}}a_{k}=f^{\left( n\right) },\text{
\ \ a.e.}  \label{1}
\end{equation}%
and%
\begin{equation*}
\sum_{k=0}^{\infty }\left\vert \mu _{k}\right\vert ^{p}<\infty .
\end{equation*}

\textit{Moreover,} 
\begin{equation*}
\left\Vert f\right\Vert _{H_{p}}\backsim \inf \left( \sum_{k=0}^{\infty
}\left\vert \mu _{k}\right\vert ^{p}\right) ^{1/p}
\end{equation*}
\textit{where the infimum is taken over all decompositions of} $f$ \textit{%
of the form} (\ref{1}).
\end{lemma}

\begin{lemma}[see e.g. \protect\cite{We3}]
\label{lemma2.2} Suppose that an operator $T$ is $\sigma $-sublinear and for
some $0<p\leq 1$%
\begin{equation*}
\int\limits_{\overset{-}{I}}\left\vert Ta\right\vert ^{p}d\mu \leq
c_{p}<\infty ,
\end{equation*}%
for every $p$-atom $a$, where $I$ denotes the support of the atom. If $T$ is
bounded from $L_{\infty \text{ }}$ to $L_{\infty },$ then 
\begin{equation*}
\left\Vert Tf\right\Vert _{p}\leq c_{p}\left\Vert f\right\Vert _{H_{p}},%
\text{ }0<p\leq 1.
\end{equation*}
\end{lemma}

\begin{lemma}[see \protect\cite{gat}]
\label{lemma2} Let $n>t,$ $t,n\in \mathbb{N}.$ Then%
\begin{equation*}
K_{M_{n}}\left( x\right) =\left\{ 
\begin{array}{c}
\frac{M_{t}}{1-r_{t}\left( x\right) },\text{ \ }x\in I_{t}\backslash I_{t+1},%
\text{ }x-x_{t}e_{t}\in I_{n}, \\ 
\frac{M_{n}-1}{2},\text{ \ \ \ \ \ \ \ \ \ \ \ \ \ \ \ \ \ \ \ \ \ \ \ \ \ \
\ \ \ }x\in I_{n}, \\ 
0,\text{ \ \ \ \ \ \ \ \ \ \ \ \ \ \ \ \ \ \ \ \ \ \ \ \ \ \ \ \ \ \
otherwise.}%
\end{array}%
\right.
\end{equation*}
\end{lemma}

For the proof of our main results we also need the following new Lemmas of
independent interest:

\begin{lemma}
\label{lemma0nn}Let $\{q_{k}:k\geq 0\}$ be a sequence of non-decreasing
numbers, satisfying condition (\ref{fn01}). \textit{Then} 
\begin{equation*}
\left\vert F_{n}\right\vert \leq \frac{c}{n}\left\{ \sum_{j=0}^{\left\vert
n\right\vert }M_{j}\left\vert K_{M_{j}}\right\vert \right\} ,
\end{equation*}%
for some positive constant $c.$
\end{lemma}

\begin{proof}
By using Abel transformation we obtain that%
\begin{equation}
Q_{n}:=\overset{n-1}{\underset{j=0}{\sum }}q_{j}=\overset{n}{\underset{j=1}{%
\sum }}q_{n-j}\cdot 1=\overset{n-1}{\underset{j=1}{\sum }}\left(
q_{n-j}-q_{n-j-1}\right) j+q_{0}n  \label{2b}
\end{equation}%
and%
\begin{equation}
F_{n}=\frac{1}{Q_{n}}\left( \overset{n-1}{\underset{j=1}{\sum }}\left(
q_{n-j}-q_{n-j-1}\right) jK_{j}+q_{0}nK_{n}\right) .  \label{2bb}
\end{equation}

Since $\{q_{k}:k\geq 0\}$ be a non-decreasing sequence, satisfying condition
(\ref{fn01}) we obtain that 
\begin{equation}
\frac{1}{Q_{n}}\left( \overset{n-1}{\underset{j=1}{\sum }}\left\vert
q_{n-j}-q_{n-j-1}\right\vert +q_{0}\right) \leq \frac{1}{Q_{n}}\left( 
\overset{n-1}{\underset{j=1}{\sum }}\left( q_{n-j}-q_{n-j-1}\right)
+q_{0}\right)  \label{2cc}
\end{equation}%
\begin{equation*}
=\frac{q_{n-1}}{Q_{n}}\leq \frac{c}{n}.
\end{equation*}

By combining (\ref{fn5}) with equalities (\ref{2bb}) and (\ref{2cc}) we
immediately get that%
\begin{equation*}
\left\vert F_{n}\right\vert \leq \left( \frac{1}{Q_{n}}\left( \overset{n-1}{%
\underset{j=1}{\sum }}\left\vert q_{n-j}-q_{n-j-1}\right\vert +q_{0}\right)
\right) \sum_{i=0}^{\left\vert n\right\vert }M_{i}\left\vert
K_{M_{i}}\right\vert \leq \frac{c}{n}\sum_{i=0}^{\left\vert n\right\vert
}M_{i}\left\vert K_{M_{i}}\right\vert .
\end{equation*}

The proof is complete by combining the estimates above.
\end{proof}

\begin{lemma}
\label{lemma00nn}Let $n\geq M_{N}$ and $\{q_{k}:k\geq 0\}$ be a sequence of
non-decreasing numbers. \textit{Then} 
\begin{equation*}
\left\vert \frac{1}{Q_{n}}\overset{n}{\underset{j=M_{N}}{\sum }}%
q_{n-j}D_{j}\right\vert \leq \frac{c}{M_{N}}\left\{ \sum_{j=0}^{\left\vert
n\right\vert }M_{j}\left\vert K_{M_{j}}\right\vert \right\} ,
\end{equation*}%
for some positive constant $c.$
\end{lemma}

\begin{proof}
Let $M_{N}\leq j\leq n.$ By using (\ref{fn5}) we get that%
\begin{equation*}
\left\vert K_{j}\right\vert \leq \frac{1}{j}\sum_{l=0}^{\left\vert
j\right\vert }M_{l}\left\vert K_{M_{l}}\right\vert \leq \frac{1}{M_{N}}%
\sum_{l=0}^{\left\vert n\right\vert }M_{l}\left\vert K_{M_{l}}\right\vert .
\end{equation*}

Let the sequence $\{q_{k}:k\geq 0\}$ be non-decreasing. Then 
\begin{equation*}
M_{N}q_{n-M_{N}-1}\leq q_{n-M_{N}-1}+q_{n-M_{N}}+...+q_{n-1}\leq Q_{n}.
\end{equation*}

If we apply (\ref{2b}) we obtain that 
\begin{equation*}
\overset{n-1}{\underset{j=M_{N}}{\sum }}\left\vert
q_{n-j}-q_{n-j-1}\right\vert j+q_{0}n\leq \overset{n-1}{\underset{j=0}{\sum }%
}\left\vert q_{n-j}-q_{n-j-1}\right\vert j+q_{0}n
\end{equation*}%
\begin{equation*}
=\overset{n-1}{\underset{j=1}{\sum }}\left( q_{n-j}-q_{n-j-1}\right)
j+q_{0}n=Q_{n}.
\end{equation*}

By using Abel transformation we find that%
\begin{equation*}
\left\vert \frac{1}{Q_{n}}\overset{n}{\underset{j=M_{N}}{\sum }}%
q_{n-j}D_{j}\right\vert =\left\vert \frac{1}{Q_{n}}\left( \overset{n-1}{%
\underset{j=M_{N}}{\sum }}\left( q_{n-j}-q_{n-j-1}\right)
jK_{j}+q_{0}nK_{n}-M_{N}q_{n-M_{N}-1}\right) \right\vert
\end{equation*}%
\begin{equation*}
\left( \frac{1}{Q_{n}}\left( \overset{n-1}{\underset{j=M_{N}}{\sum }}%
\left\vert q_{n-j}-q_{n-j-1}\right\vert j+q_{0}n+M_{N}q_{n-M_{N}-1}\right)
\right) \frac{1}{M_{N}}\sum_{i=0}^{\left\vert n\right\vert }M_{i}\left\vert
K_{M_{i}}\right\vert \leq \frac{1}{M_{N}}\sum_{i=0}^{\left\vert n\right\vert
}M_{i}\left\vert K_{M_{i}}\right\vert .
\end{equation*}%
The proof is complete.
\end{proof}

\begin{lemma}
\label{lemma5}\ Let $\{q_{k}:k\geq 0\}$ be a sequence of non-decreasing
numbers, satisfying condition (\ref{fn01}). Let $x\in I_{N}^{k,l},$ $%
k=0,\dots ,N-2,$ $l=k+1,\dots ,N-1.$ Then 
\begin{equation*}
\int_{I_{N}}\left\vert F_{n}\left( x-t\right) \right\vert d\mu \left(
t\right) \leq \frac{cM_{l}M_{k}}{nM_{N}}.
\end{equation*}

Let $x\in I_{N}^{k,N},$ $k=0,\dots ,N-1.$ Then%
\begin{equation*}
\int_{I_{N}}\left\vert F_{n}\left( x-t\right) \right\vert d\mu \left(
t\right) \leq \frac{cM_{k}}{M_{N}}.
\end{equation*}%
Here $c$ is a positive constant.
\end{lemma}

\begin{proof}
Let $x\in I_{N}^{k,l}$. Then, by applying Lemma \ref{lemma2}, we have that 
\begin{equation}
K_{M_{n}}\left( x\right) =0,\,\,\text{when \thinspace \thinspace }n>l.
\label{777a}
\end{equation}

Let $k<n\leq l$. Then we get that%
\begin{equation}
\left\vert K_{M_{n}}\left( x\right) \right\vert \leq cM_{k}.  \label{777777}
\end{equation}

Let $x\in I_{N}^{k,l},$ for $0\leq k<l\leq N-1$ and $t\in I_{N}.$ Since $%
x-t\in $ $I_{N}^{k,l}$ and $n\geq M_{N},$ by combining Lemma \ref{lemma0nn}
with (\ref{777a}) and (\ref{777777}), we obtain that%
\begin{equation}
\int_{I_{N}}\left\vert F_{n}\left( x-t\right) \right\vert d\mu \left(
t\right) \leq \frac{c}{n}\underset{i=0}{\overset{\left\vert n\right\vert }{%
\sum }}M_{i}\int_{I_{N}}\left\vert K_{M_{i}}\left( x-t\right) \right\vert
d\mu \left( t\right)  \label{888}
\end{equation}%
\begin{equation*}
\leq \frac{c}{n}\int_{I_{N}}M_{i}\overset{l}{\underset{i=0}{\sum }}M_{k}d\mu
\left( t\right) \leq \frac{cM_{k}M_{l}}{nM_{N}}
\end{equation*}%
and the first astimate is proved.

Now, let $x\in I_{N}^{k,N}$. Since $x-t\in I_{N}^{k,N}$ for $t\in I_{N},$ by
applying Lemma \ref{lemma2}, we obtain that 
\begin{equation*}
\left\vert K_{M_{i}}\left( x-t\right) \right\vert \leq cM_{k},\text{ }\left(
k\in 
\mathbb{N}
\right) .
\end{equation*}%
Hence, according to Lemma \ref{lemma0nn}, we have that%
\begin{equation}
\int_{I_{N}}\left\vert F_{n}\left( x-t\right) \right\vert d\mu \left(
t\right) \leq \frac{c}{n}\underset{i=0}{\overset{\left\vert n\right\vert }{%
\sum }}M_{i}\int_{I_{N}}\left\vert K_{M_{i}}\left( x-t\right) \right\vert
d\mu \left( t\right)  \label{11110}
\end{equation}%
\begin{equation*}
\leq \frac{c}{n}\overset{\left\vert n\right\vert -1}{\underset{i=0}{\sum }}%
M_{i}\int_{I_{N}}M_{k}d\mu \left( t\right) \leq \frac{cM_{k}}{M_{N}}.
\end{equation*}

By combining (\ref{888}) and (\ref{11110}) we complete the proof of \ Lemma %
\ref{lemma5}.
\end{proof}

Analogously we can prove the similar estimation, but now without any
restriction like (\ref{fn01}).

\begin{lemma}
\label{lemma5aaa}\ Let $x\in I_{N}^{k,l},$ $k=0,\dots ,N-1,$ $l=k+1,\dots ,N$
and $\{q_{k}:k\geq 0\}$ be a sequence of non-decreasing sequence. Then%
\begin{equation*}
\int_{I_{N}}\left\vert \frac{1}{Q_{n}}\overset{n}{\underset{j=M_{N}}{\sum }}%
q_{n-j}D_{j}\left( x-t\right) \right\vert d\mu \left( t\right) \leq \frac{%
cM_{l}M_{k}}{M_{N}^{2}},
\end{equation*}%
for some positive constant $c.$
\end{lemma}

\section{Proofs of the Theorems}

\begin{proof}[Proof of Theorem \protect\ref{theorem2fejerstrong}]
By Lemma \ref{lemma2.1}, the proof of Theorem \ref{theorem2fejerstrong} will
be complete, if we show that%
\begin{equation}
\frac{1}{\log ^{\left[ 1/2+p\right] }n}\overset{n}{\underset{m=1}{\sum }}%
\frac{\left\Vert t_{m}a\right\Vert _{H_{p}}^{p}}{m^{2-2p}}\leq c_{p},
\label{14c}
\end{equation}%
for every $p$-atom $a,$ with support$\ I$, $\mu \left( I\right) =M_{N}^{-1}.$
We may assume that $I=I_{N}.$ It is easy to see that $S_{n}\left( a\right)
=t_{n}\left( a\right) =0,$ when $n\leq M_{N}$. Therefore, we can suppose
that $n>M_{N}$.

Let $x\in I_{N}.$ Since $t_{n}$ is bounded from $L_{\infty }$ to $L_{\infty
} $ (the boundedness follows from (\ref{fn6})) and $\left\Vert a\right\Vert
_{\infty }\leq M_{N}^{1/p}$ we obtain that 
\begin{equation*}
\int_{I_{N}}\left\vert t_{m}a\right\vert ^{p}d\mu \leq \frac{\left\Vert
a\right\Vert _{\infty }^{p}}{M_{N}}\leq c<\infty ,\text{ \ }0<p\leq 1/2.
\end{equation*}%
Hence, 
\begin{equation}
\frac{1}{\log ^{\left[ 1/2+p\right] }n}\overset{n}{\underset{m=1}{\sum }}%
\frac{\int_{I_{N}}\left\vert t_{m}a\right\vert ^{p}d\mu }{m^{2-2p}}\leq 
\frac{1}{\log ^{\left[ 1/2+p\right] }n}\overset{n}{\underset{k=1}{\sum }}%
\frac{1}{m^{2-2p}}\leq c<\infty .  \label{14b}
\end{equation}

It is easy to see that 
\begin{equation}
\left\vert t_{m}a\left( x\right) \right\vert =\int_{I_{N}}\left\vert a\left(
t\right) F_{n}\left( x-t\right) \right\vert d\mu \left( t\right)
=\int_{I_{N}}\left\vert a\left( t\right) \frac{1}{Q_{n}}\overset{n}{\underset%
{j=M_{N}}{\sum }}q_{n-j}D_{j}\left( x-t\right) \right\vert d\mu \left(
t\right)   \label{14a}
\end{equation}%
\begin{equation*}
\leq \left\Vert a\right\Vert _{\infty }\int_{I_{N}}\left\vert \frac{1}{Q_{n}}%
\overset{n}{\underset{j=M_{N}}{\sum }}q_{n-j}D_{j}\left( x-t\right)
\right\vert d\mu \left( t\right) \leq M_{N}^{1/p}\int_{I_{N}}\left\vert 
\frac{1}{Q_{n}}\overset{n}{\underset{j=M_{N}}{\sum }}q_{n-j}D_{j}\left(
x-t\right) \right\vert d\mu \left( t\right) 
\end{equation*}

Let $t_{n}$ be Nörlund means, with non-decreasing coefficients $%
\{q_{k}:k\geq 0\}$ and $x\in I_{N}^{k,l},\,0\leq k<l\leq N.$ Then, in the
view of Lemma \ref{lemma5aaa} we get that 
\begin{equation}
\left\vert t_{m}a\left( x\right) \right\vert \leq cM_{l}M_{k}M_{N}^{1/p-2},%
\text{ for }0<p\leq 1/2.  \label{12q1}
\end{equation}

First, we consider the case $0<p<1/2.$ \ By using (\ref{1.1}), (\ref{14a}), (%
\ref{12q1}) we find that%
\begin{equation}
\int_{\overline{I_{N}}}\left\vert t_{m}a\right\vert ^{p}d\mu =\overset{N-2}{%
\underset{k=0}{\sum }}\overset{N-1}{\underset{l=k+1}{\sum }}%
\sum\limits_{x_{j}=0,\text{ }j\in \{l+1,\dots
,N-1\}}^{m_{j-1}}\int_{I_{N}^{k,l}}\left\vert t_{m}a\right\vert ^{p}d\mu +%
\overset{N-1}{\underset{k=0}{\sum }}\int_{I_{N}^{k,N}}\left\vert
t_{m}a\right\vert ^{p}d\mu  \label{7aaa}
\end{equation}%
\begin{equation}
\leq c\overset{N-2}{\underset{k=0}{\sum }}\overset{N-1}{\underset{l=k+1}{%
\sum }}\frac{m_{l+1}\dotsm m_{N-1}}{M_{N}}\left( M_{l}M_{k}\right)
^{p}M_{N}^{1-2p}+\overset{N-1}{\underset{k=0}{\sum }}\frac{1}{M_{N}}%
M_{k}^{p}M_{N}^{1-p}  \notag
\end{equation}%
\begin{equation*}
\leq cM_{N}^{1-2p}\overset{N-2}{\underset{k=0}{\sum }}\overset{N-1}{\underset%
{l=k+1}{\sum }}\frac{\left( M_{l}M_{k}\right) ^{p}}{M_{l}}+\overset{N-1}{%
\underset{k=0}{\sum }}\frac{M_{k}^{p}}{M_{N}^{p}}\leq cM_{N}^{1-2p}.
\end{equation*}

Moreover, according to (\ref{7aaa}), we get that%
\begin{equation*}
\overset{\infty }{\underset{m=M_{N}+1}{\sum }}\frac{\int_{\overline{I_{N}}%
}\left\vert t_{m}a\right\vert ^{p}d\mu }{m^{2-2p}}\leq \overset{\infty }{%
\underset{m=M_{N}+1}{\sum }}\frac{cM_{N}^{1-2p}}{m^{2-2p}}<c<\infty ,\text{
\ }\left( 0<p<1/2\right) .
\end{equation*}%
Now, by combining this estimate with (\ref{14b}) we obtain (\ref{14c}) so
the proof of part a) is complete.

Let $p=1/2$ and $t_{n}$ be Nörlund means, with non-decreasing coefficients $%
\{q_{k}:k\geq 0\}$, satisfying condition (\ref{fn01}). We can write that 
\begin{equation}
\left\vert t_{m}a\left( x\right) \right\vert \leq \int_{I_{N}}\left\vert
a\left( t\right) \right\vert \left\vert F_{m}\left( x-t\right) \right\vert
d\mu \left( t\right)  \label{14aa}
\end{equation}%
\begin{equation*}
\leq \left\Vert a\right\Vert _{\infty }\int_{I_{N}}\left\vert F_{m}\left(
x-t\right) \right\vert d\mu \left( t\right) \leq
M_{N}^{2}\int_{I_{N}}\left\vert F_{m}\left( x-t\right) \right\vert d\mu
\left( t\right) .
\end{equation*}

Let $x\in I_{N}^{k,l},\,0\leq k<l<N.$ Then, in the view of Lemma \ref{lemma5}
we get that 
\begin{equation}
\left\vert t_{m}a\left( x\right) \right\vert \leq \frac{cM_{l}M_{k}M_{N}}{m}.
\label{13q1}
\end{equation}

Let $x\in I_{N}^{k,N}.$ Then, according to Lemma \ref{lemma5} we obtain that 
\begin{equation}
\left\vert t_{m}a\left( x\right) \right\vert \leq cM_{k}M_{N}.  \label{13q2}
\end{equation}

By combining (\ref{1.1}), (\ref{14aa}), (\ref{13q1}) and (\ref{13q2}) we
obtain that%
\begin{equation*}
\int_{\overline{I_{N}}}\left\vert t_{m}a\left( x\right) \right\vert
^{1/2}d\mu \left( x\right)
\end{equation*}%
\begin{equation}
\leq c\overset{N-2}{\underset{k=0}{\sum }}\overset{N-1}{\underset{l=k+1}{%
\sum }}\frac{m_{l+1}\dotsm m_{N-1}}{M_{N}}\frac{\left( M_{l}M_{k}\right)
^{1/2}M_{N}^{1/2}}{m^{1/2}}+\overset{N-1}{\underset{k=0}{\sum }}\frac{1}{%
M_{N}}M_{k}^{1/2}M_{N}^{1/2}  \notag
\end{equation}%
\begin{equation*}
\leq M_{N}^{1/2}\overset{N-2}{\underset{k=0}{\sum }}\overset{N-1}{\underset{%
l=k+1}{\sum }}\frac{\left( M_{l}M_{k}\right) ^{1/2}}{m^{1/2}M_{l}}+\overset{%
N-1}{\underset{k=0}{\sum }}\frac{M_{k}^{1/2}}{M_{N}^{1/2}}\leq \frac{%
cM_{N}^{1/2}N}{m^{1/2}}+c.
\end{equation*}

It follows that%
\begin{equation}
\frac{1}{\log n}\overset{n}{\underset{m=M_{N}+1}{\sum }}\frac{\int_{%
\overline{I_{N}}}\left\vert t_{m}a\left( x\right) \right\vert ^{1/2}d\mu
\left( x\right) }{m}\leq \frac{1}{\log n}\overset{n}{\underset{m=M_{N}+1}{%
\sum }}\left( \frac{cM_{N}^{1/2}N}{m^{3/2}}+\frac{c}{m}\right) <c<\infty .
\label{15b}
\end{equation}

The proof of part b) is completed by just combining (\ref{14b}) and (\ref%
{15b}).
\end{proof}

\begin{proof}[\textbf{Proof of Theorem \protect\ref{theorem3fejermax2}}]
Since $t_{n}$ is bounded from $L_{\infty }$ to $L_{\infty }$ (the
boundedness follows from (\ref{fn6})), by Lemma \ref{lemma2.2}, the proof of
Theorem \ref{theorem3fejermax2} will be complete, if we show that%
\begin{equation*}
\int\limits_{\overline{I}_{N}}\left( \underset{n\in \mathbb{N}}{\sup }\frac{%
\left\vert t_{n}a\right\vert }{\log ^{2\left[ 1/2+p\right] }\left(
n+1\right) \left( n+1\right) ^{1/p-2}}\right) ^{p}d\mu \leq c<\infty 
\end{equation*}%
for every p-atom $a,$ where $I$ denotes the support of the atom$.$ Let $a$
be an arbitrary p-atom, with support$\ I$ and $\mu \left( I\right)
=M_{N}^{-1}.$ Analogously to in the proof of Theorem \ref%
{theorem2fejerstrong} we may assume that $I=I_{N}$ and $n>M_{N}$.

Let $x\in I_{N}^{k,l},\,0\leq k<l\leq N.$ Then, by combining (\ref{14a}) and
Lemma \ref{lemma5aaa}, (see also (\ref{12q1})) we get that 
\begin{equation}
\frac{\left\vert t_{n}\left( a\left( x\right) \right) \right\vert }{\left(
n+1\right) ^{1/p-2}\log ^{2\left[ 1/2+p\right] }\left( n+1\right) }\leq 
\frac{M_{N}^{1/p}}{M_{N}^{1/p-2}N^{2\left[ 1/2+p\right] }}%
\int_{I_{N}}\left\vert \frac{1}{Q_{n}}\overset{n}{\underset{j=M_{N}}{\sum }}%
q_{n-j}D_{j}\left( x-t\right) \right\vert d\mu \left( t\right)   \label{1112}
\end{equation}%
\begin{equation*}
\leq \frac{cM_{N}^{1/p}}{M_{N}^{1/p-2}N^{2\left[ 1/2+p\right] }}\frac{%
M_{l}M_{k}}{M_{N}^{2}}=\frac{cM_{l}M_{k}}{N^{2\left[ 1/2+p\right] }}.
\end{equation*}%
By combining (\ref{1.1}) and (\ref{1112}) we obtain that (see \cite{tep2}
and \cite{tep3}) 
\begin{equation*}
\int_{\overline{I_{N}}}\left\vert t^{\ast }a\right\vert ^{p}d\mu =\overset{%
N-2}{\underset{k=0}{\sum }}\overset{N-1}{\underset{l=k+1}{\sum }}%
\sum\limits_{x_{j}=0,j\in
\{l+1,...,N-1\}}^{m_{j-1}}\int_{I_{N}^{k,l}}\left\vert t^{\ast }a\right\vert
^{p}d\mu +\overset{N-1}{\underset{k=0}{\sum }}\int_{I_{N}^{k,N}}\left\vert
t^{\ast }a\right\vert ^{p}d\mu 
\end{equation*}%
\begin{equation*}
\leq \overset{N-2}{\underset{k=0}{\sum }}\overset{N-1}{\underset{l=k+1}{\sum 
}}\frac{m_{l+1}...m_{N-1}}{M_{N}}\left( \frac{M_{l}M_{k}}{N^{2\left[ 1/2+p%
\right] }}\right) ^{p}+\overset{N-1}{\underset{k=0}{\sum }}\frac{1}{M_{N}}%
\left( \frac{M_{N}M_{k}}{N^{2\left[ 1/2+p\right] }}\right) ^{p}
\end{equation*}%
\begin{equation*}
\leq \frac{c}{N^{2\left[ 1/2+p\right] }}\overset{N-2}{\underset{k=0}{\sum }}%
\overset{N-1}{\underset{l=k+1}{\sum }}\frac{\left( M_{l}M_{k}\right) ^{p}}{%
M_{l}}+\frac{c}{M_{N}^{1-2p}N^{2p\left[ 1/2+p\right] }}\overset{N-1}{%
\underset{k=0}{\sum }}\frac{M_{k}^{p}}{M_{N}^{p}}<\infty .
\end{equation*}

The proof is complete.
\end{proof}

\textbf{A final remark: }\ Several of the operators considered in this
paper, e.g. those described by the Nörlund means are called Hardy type
operators in the literature. The mapping properties of such operators,
especially between weighted Lebegue spaces, is much studied in the
literature, see e.g. the books \cite{kp} and \cite{mkp} and the references
there. Such complimentary information can be of interest for further studies
of the inequalities considered in this paper.

\end{document}